\documentclass[runningheads]{llncs}
\usepackage[T1]{fontenc}

\usepackage{graphicx}

\usepackage{amsfonts,amscd,amsmath, amssymb}
\usepackage{tikz-cd} 
\usepackage[shortlabels]{enumitem}
\newcommand{\Sup}{\langle \sup\rangle }
\newcommand{\Min}{\langle \mathrm{mub}\rangle }
\newcommand{\Past}{\langle \textnormal{\textsc{P}}\rangle }

\begin{document}
\title{The Modal Logic of Minimal Upper Bounds}

\author{Søren Brinck Knudstorp\orcidID{0009-0008-9835-4195}
}
\authorrunning{S. B. Knudstorp}
\institute{ILLC \& Philosophy, University of Amsterdam, Amsterdam, the Netherlands \\ \footnotesize{Accepted for publication in the postproceedings of TbiLLC 2023.}\vspace{-.4cm}}
\maketitle              
\begin{abstract}

To formalize patterns of information increase and decrease, Van Benthem (1996) proposed modal information logic ($\mathbf{MIL}$), a modal logic over partial orders. In $\mathbf{MIL}$, points are interpreted as information states and least upper bounds, when existent, as informational sums. A natural counterpart to this logic is the modal logic of minimal upper bounds ($\mathbf{MIN}$), interpreting \textit{minimal}, rather than \textit{least}, upper bounds as informational sums. 

This paper presents the logic $\mathbf{MIN}$, and in the main result, it is shown that the modal language cannot distinguish the two interpretations: a formula is valid in $\mathbf{MIN}$ if and only if it is valid in $\mathbf{MIL}$.
Leveraging the work of \cite{SBKMIL}, as corollaries, an axiomatization of $\mathbf{MIN}$ and a proof of decidability are obtained.

\end{abstract}
%
%
%
\section{Introduction}
A nuanced definitional distinction is that between minimal and least upper bounds. Given a partial order $(P, \leq)$, an element $x\in P$ is a minimal upper bound of a subset $Y\subseteq P$ if it is both an upper bound and minimal w.r.t. this property, i.e., no other upper bound is below $x$; it is a least upper bound (supremum) if it is both an upper bound and least w.r.t. this property, i.e., all other upper bounds are above $x$. Despite sounding synonymous, the conditions for minimal and least upper bounds diverge, as exemplified in Figure \ref{fig:M1}.

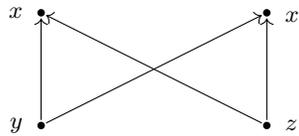
\begin{figure}[h!tbp]
\begin{center}
\begin{tikzpicture}[scale=1.5, label distance=2pt, shorten >=.5pt, shorten <=.5pt]

   \node [fill, circle, inner sep=1pt, label=left:$y$, label=right:{}] (worldu) at (-1,1) {};
   \node [fill, circle, inner sep=1pt, label=right:$x'$, label=left:{}] (worldw) at (1,2) {};
   \node [fill, circle, inner sep=1pt, label=left:$x$, label=right:{}] (worldm) at (-1,2) {};
   \node [fill, circle, inner sep=1pt, label=right:$z$, label=left:{}] (worldv) at (1,1) {};
   
  \draw [->] (worldu) to (worldw);
    \draw [->] (worldu) to (worldm);
    \draw [->] (worldv) to (worldw);
  \draw [->] (worldv) to (worldm);
\end{tikzpicture}
\end{center} \vspace{-.4cm}
\caption{$x$ is an upper bound of $\{y,z\}$ and minimal in this regard, as no other upper bound is below $x$. However, $x$ is not least in this regard, as $x'$ is another upper bound that is not above $x$. Thus, $\{y,z\}$ has minimal upper bounds, namely $x$ and $x'$, but no least upper bound. This example also illustrates that multiple minimal upper bounds may exist, unlike least upper bounds which, when existent, are unique due to antisymmetry. 
} \label{fig:M1}
\end{figure}

\cite{JvB1996} introduced a modal logic, termed Modal Information Logic ($\mathbf{MIL}$), on partial orders $(P, \leq)$ featuring a binary modality `$\Sup$' with semantics
\begin{align*}
    x\Vdash \Sup\varphi \psi\qquad \textbf{iff}\qquad \text{$\exists y,z\in P$ such that $y\Vdash \varphi$; $z\Vdash \psi$; and $x=\sup\{y,z\}$},
\end{align*}
where $x=\sup\{y,z\}$ indicates that $x$ is the least upper bound, the supremum, of $\{y,z\}$. 

Inspired by this logic, we explore a natural analogue where the semantics of the modality refers to the induced minimal-upper-bound relation rather than the least-upper-bound relation. We denote this new logic of minimal upper bounds by $\mathbf{MIN}$. While $\mathbf{MIN}$ retains the same modality symbol, we will typically designate it as `$\Min$' to emphasize the distinct semantic clause:
\begin{align*}
    x\Vdash \Min\varphi \psi\qquad \textbf{iff}\qquad \text{$\exists y,z\in P$ such that $y\Vdash \varphi$; $z\Vdash \psi$; and $x\in\mathrm{mub}\{y,z\}$},
\end{align*}
where $\mathrm{mub}\{y,z\}$ denotes the set of minimal upper bounds of $\{y,z\}$. 

This paper's inquiry centers on understanding $\mathbf{MIN}$ and especially its relationship with $\mathbf{MIL}$. Specifically, we're intrigued by the problem of axiomatization: Can we identify modal principles that are valid under one interpretation but not under the other?\footnote{Note that in Figure \ref{fig:M1}, if $y\Vdash\varphi$ and $z\Vdash\psi$, it follows that $x\Vdash \Min \varphi \psi$, but not that $x\Vdash \Sup \varphi \psi$. So clearly, there are differences at the model level, but what, if any, differences are there at the level of validities and consequences?}

In addition to this mathematically motivated exploration of a modal perspective on minimal versus least upper bounds, there is a philosophical impetus for studying $\mathbf{MIN}$. The name `modal \textit{information} logic' is not arbitrary. Conceiving the points in the partial order $(P, \leq)$ as information states, the order as an arrangement of, or inclusion of, information, and the modality as describing `merge' or `fusion' of information, $\mathbf{MIL}$ serves as a qualitative yet rigorous theory of information capturing patterns of information aggregation. $\mathbf{MIN}$ shares this informational interpretation. However, unlike $\mathbf{MIL}$, which interprets the modality in terms of suprema, entailing uniqueness of fusions, $\mathbf{MIN}$ formalizes settings where states can have multiple, incomparable fusions. Consequently, $\mathbf{MIN}$ is aptly described as the (modal information) logic of incomparable fusions.

The main contribution of this paper is to show that, perhaps unexpectedly, these two logical systems coincide: at the level of validities and consequences, there are no modal principles telling the interpretations apart. The inclusion $\mathbf{MIL}\subseteq \mathbf{MIN}$ follows easily, as the axiomatization of $\mathbf{MIL}$ provided in \cite{SBKMIL} can be seen sound w.r.t. $\mathbf{MIN}$. The converse inclusion $\mathbf{MIN}\subseteq \mathbf{MIL}$ proves to be more tricky. It is established through a representation construction that, at its core, exploits that an upper-bounded set may lack a least upper bound not only due to the existence of multiple minimal upper bounds, as was the case in Figure \ref{fig:M1}, but also by the presence of an infinite descending chain of upper bounds. As a consequence of this result, we additionally obtain an axiomatization of $\mathbf{MIN}$, as well as a proof of decidability, leveraging that \cite{SBKMIL} axiomatized $\mathbf{MIL}$ and proved it to be decidable.

Finally, we explore some generalizations, notably showing that even upon augmenting the logics with $\Min$ and $\Sup$'s respective residual implications, their consequence relations remain equal.
\\\\
The paper is organized as follows: 
\begin{itemize}
    \item Section \ref{sec:defining} provides preliminaries and clarifies connections with relevant logical systems.
    \item Section \ref{sec:minComp} constitutes the core of the paper, presenting the principal lemma that allows us to deduce that $\mathbf{MIN}$ and $\mathbf{MIL}$ coincide.
    \item Section \ref{sec:Extensions} concludes the paper with further results, including analogous results for the logics enriched with the residual implication.
\end{itemize}

\section{Definitions and Supporting Results}\label{sec:defining}
In this section, we lay the groundwork by presenting preliminary concepts, key definitions, and supporting results. We also contextualize $\mathbf{MIN}$ within the broader landscape of logical systems.

We start by providing some pertinent definitions.

\begin{definition}[Language]
    The \textnormal{formulas} $\varphi\in \mathcal{L}$ of our language $\mathcal{L}$ are given by the BNF-grammar:
        $$\varphi\mathrel{::=} \bot\hspace{0.1cm}|\hspace{0.1cm}p \hspace{0.1cm}| \hspace{0.1cm} \neg\varphi\hspace{0.1cm}|\hspace{0.1cm}\varphi\lor\varphi\hspace{0.1cm}|\hspace{0.1cm}\Min\varphi\varphi,$$
    where $p\in \mathbf{Prop}$ is a propositional letter, `$\Min$' a binary modality symbol, and $\bot$ the falsum constant.
\end{definition}

\begin{definition}[Frames and models]
    A \textnormal{{poset (frame)}} is a pair $\mathbb{F}=(P,\leq)$, where $P$ is a set and $\leq$ is a partial order on $P$ (i.e., reflexive, transitive and antisymmetric).

    A \textnormal{{poset model}} is a triple $\mathbb{M}=(P,\leq, V)$, where $(P,\leq)$ is a poset (frame), and $V:\mathbf{Prop}\to \mathcal{P}(P)$ is a \textnormal{valuation} on $P$.
\end{definition}
To ensure clarity, before providing the semantics, let's lay down the distinction between \textit{a minimal} and \textit{the least} upper bound:
\begin{definition}[Least and minimal upper bound]
    Let $(P, \leq)$ be a partial order  and $s,t,u\in P$. We say that $s$ is \textnormal{a minimal upper bound} of $\{t,u\}$ and write $s\in \mathrm{mub}\{t,u\}$ if 
    \begin{itemize}
        \item $s$ is an upper bound of $\{t,u\}$, i.e., $t\leq s$ and $u\leq s$; and
        \item $s$ is minimal with this property, i.e., if $x$ is an upper bound of $\{t,u\}$, then $x\not < s$.
    \end{itemize}
    We say that $s$ is \textnormal{the least upper bound} of $\{t,u\}$ and write $s= \sup\{t,u\}$ if 
    \begin{itemize}
        \item $s$ is an upper bound of $\{t,u\}$; and
        \item $s$ is least with this property, i.e., if $x$ is an upper bound of $\{t,u\}$, then $s\leq x$.
    \end{itemize}
\end{definition}
\begin{definition}[Semantics]
    Let $\mathbb{M}=(P, \leq, V)$ be a poset model and $s\in P$. \textnormal{Satisfaction} of a formula $\varphi\in \mathcal{L}$ at $s$ in $\mathbb{M}$ (denoted $\mathbb{M}, s\Vdash \varphi$) is defined as follows:
    \begin{align*}
        &\mathbb{M}, s\nVdash \bot, && \textit{ } && \textit{ }\\
        &\mathbb{M}, s\Vdash p && \textbf{iff} && s\in V(p),\\
        &\mathbb{M}, s\Vdash \neg\varphi && \textbf{iff} && \mathbb{M}, s\nVdash \varphi,\\
        &\mathbb{M}, s\Vdash \varphi\lor\psi && \textbf{iff} && \mathbb{M}, s\Vdash \varphi \textit{\hspace{0.1cm} or \hspace{0.1cm}} \mathbb{M}, s\Vdash \psi,\\
        &\mathbb{M}, s\Vdash \Min\varphi\psi && \textbf{iff} && \text{there exist $t,u\in P$ such that } \mathbb{M}, t\Vdash \varphi,\textit{ } \mathbb{M}, u\Vdash \psi, \\
        & && && \hspace{4.5cm}\text{and $s\in \mathrm{mub}\{t,u\}$.}
    \end{align*}
\textit{Validity} of a formula $\varphi\in \mathcal{L}$ in a frame $\mathbb{F}$ (denoted $\mathbb{F}\Vdash \varphi$) is defined as usual.
\end{definition}

With these definitions in place, we get two logics, depending on the semantics of the modality:
\begin{definition}[Logics]
    \textnormal{The modal logic of minimal upper bounds} is denoted \upshape{$\mathbf{MIN}$} \itshape and defined as the set of validities (or the consequence relation) over the class of poset frames in the given language with the provided semantics.
    
    Likewise, $\mathbf{MIL}$ denotes the standard modal information logic of suprema on posets, but where the modality symbol---although the same---is expressed as \upshape `$\langle\sup\rangle$' \itshape for ease of understanding and interpreted with respect to the induced supremum relation instead; that is, 
    \begin{align*}
        &\mathbb{M}, s\Vdash \Sup\varphi\psi && \textbf{iff} && \text{there exist $t,u\in P$ such that } \mathbb{M}, t\Vdash \varphi,\textit{ } \mathbb{M}, u\Vdash \psi, \\
        & && && \hspace{4.5cm}\text{and $s= \sup\{t,u\}$.}
    \end{align*}
\end{definition}

\begin{definition}[Satisfaction notation]
    To emphasize the semantic clause for the modality applied, we may subscript $\Vdash$ with $_M$ or $_S$. $\mathbb{M}, s\Vdash_M \varphi$ indicates satisfaction under the m.u.b.  clause, while $\mathbb{M}, s\Vdash_S \varphi$ denotes satisfaction under the sup clause.
\end{definition}

\begin{remark}[Extension of $\mathbf{S4}$]
    It is not too hard to verify that the `past-looking' unary modality `$\Past$' with semantics
        $$\mathbb{M}, s\Vdash \Past\varphi\qquad \textbf{iff}\qquad \text{there exists $t\leq s$ such that $\mathbb{M}, t\Vdash \varphi$,}$$
    is definable as
        $$\Past\varphi\mathrel{:=}\Min \varphi \top.$$
    Thus, being defined over posets, $\mathbf{MIN}$ emerges as an extension of $\mathbf{S4}$, enriching it with additional expressive power. Likewise for $\mathbf{MIL}$, the past-looking modality is definable as $\Past\varphi\mathrel{:=}\Sup \varphi \top.$

    As $\mathbf{S4}$ is the logic of general preorders as well, this connection prompts exploring analogues of $\mathbf{MIN}$ and $\mathbf{MIL}$ defined over preorders rather than partial orders. While we defer this exploration for now, it will be revisited in Section \ref{sec:Extensions}.
\end{remark}

\begin{remark}[Belated logical systems]
    The idea of semantic `fusion clauses' isn't exclusive to $\mathbf{MIN}$ and $\mathbf{MIL}$. Similar clauses are found in various other logical systems, such as conjunction in truthmaker semantics  (\cite{Fraassen69}, \cite{Fine17}); intensional conjunction in relevant logics (\cite{AndersonBelnap75}, \cite{Urquhart72}); tensor or split disjunction in team semantical settings (\cite{Yang16}, \cite{Aloni22}); and the product connective in the Lambek calculus (\cite{Lambek58}, \cite{Buszkowski21}).

    For further study of the connections with truthmaker semantics and the Lambek calculus, we direct the reader to \cite{JvB2019} \& \cite{SBKTML} and \cite{SBKMIL}, respectively.
\end{remark}
Having put forth the logics and situated them within a broader logical context, let's continue explaining the results achieved. The principal proof of the present research shows that 
    $$\mathbf{MIN}\subseteq \mathbf{MIL}$$ 
through representation. Essentially, this achieves three results in one go, namely: 
\begin{enumerate}
    \item
$\mathbf{MIN}=\mathbf{MIL}$; 
    \item an axiomatization of $\mathbf{MIN}$; and 
    \item a proof that $\mathbf{MIN}$ is decidable.
\end{enumerate}
These results follow because (i) \cite{SBKMIL} axiomatizes $\mathbf{MIL}$ and proves it decidable, and (ii) the converse inclusion $\mathbf{MIL}\subseteq \mathbf{MIN}$ can be seen to hold as well, as we will now demonstrate by proving $\mathbf{MIL}$'s axiomatization sound with respect to $\mathbf{MIN}$.

\begin{theorem}[Axiomatization of $\mathbf{MIL}$ \cite{SBKMIL}]\label{MILAx}
    $\mathbf{MIL}$ is sound and complete w.r.t. the least normal modal logic with axioms:
    \begin{enumerate}[leftmargin=40pt]
        \item [\textnormal{(Re.)}] $p\land q\to \Sup pq$
        \item [\textnormal{(4)}] $\Past \Past p\to \Past p$
        \item [\textnormal{(Co.)}] $\Sup pq\rightarrow \Sup qp$
        \item [\textnormal{(Dk.)}] $(p\land \Sup qr)\to \Sup pq$
    \end{enumerate}
\end{theorem}

\begin{theorem}\label{Soundness}
    $\mathbf{MIL}\subseteq \mathbf{MIN}$    
\end{theorem}
\begin{proof}
    Routine check that $\mathbf{MIN}$ is a normal modal logic validating the axioms (Re.), (4), (Co.), and (Dk.) under the minimal-upper-bound interpretation of the modality.
\end{proof}

\section{Proof of the Principal Lemma and Main Results}\label{sec:minComp}
We proceed by outlining the general strategy of, and some intuition for, our proof of the other inclusion $\mathbf{MIN}\subseteq\mathbf{MIL}$.
\paragraph{Proof Strategy.}\label{strategy} Note that it suffices to show that for any formula $\psi$, model $\mathbb{M}=(P, \leq, V)$ and state $s\in P$, if 
    $$\mathbb{M}, s\nVdash_S \psi,$$
then there are some $\mathbb{M'}, s'$ s.t.
    $$\mathbb{M'}, s'\nVdash_M \psi.$$
The basic idea for proving this is a proof by representation where we extend any given poset model $\mathbb{M}=(P, \leq, V)$ to a poset model $\mathbb{M'}=(P', \leq', V')\supseteq (P, \leq, V)$ s.t. the following two conditions are met:
\begin{align*}
    &(1) &&\forall s\in P, \varphi\in \mathcal{L}:\qquad  \mathbb{M}, s\Vdash_S\varphi && \Leftrightarrow && \mathbb{M}', s\Vdash_S\varphi\\
    &(2) &&\forall s', t', u'\in P': \qquad s'\in\mathrm{mub}\{t', u'\} && \Leftrightarrow && s'=\sup\{t', u'\}
\end{align*}
Thus, on $\mathbb{M}'$, the relations $\Vdash_S$ and $\Vdash_M$ coincide, and we've obtained our $\Vdash_M$-refuting model. 

Before delving into the technical intricacies of this representation construction, let's first develop some intuition behind it. 
\paragraph{Intuition for Representation Construction.} The crux of the proof strategy lies in obtaining a poset that satisfies clause (2). This approach is motivated by the observation that for any poset $(P, \leq)$ and any elements $s,t,u\in P$, one implication of (2) holds true:
\begin{align*}
    &s\in\mathrm{mub}\{t, u\} && \Leftarrow && s=\sup\{t, u\}
\end{align*}
Thus, our problem are failures of the converse direction: when $s$ is a minimal upper bound of $\{t, u\}$ but not its supremum.
In accordance with clause (1), which requires preservation of $\Vdash_S$, we seek to address such cases not by making $s$ be the least upper bound of $\{t,u\}$, but by ensuring that it no longer is a minimal upper bound.

Crucially, we will exploit the fact that there are two ways for an upper bounded set $\{t,u\}$ to lack a supremum: 
\begin{enumerate}
    \item [(i)] Incomparable upper bounds; vs. 
    \item [(ii)] Infinitely descending chain(s) of upper bounds.
\end{enumerate}
In essence, the idea is to transform cases of (i) into (also being) cases of (ii). By doing so, we convert scenarios where $\sup\{t, u\} \neq s \in \mathrm{mub}\{t, u\}$ into instances where $s$ no longer serves as a minimal upper bound, thus achieving the sought bi-implication. In spirit, we will be transforming as depicted below:

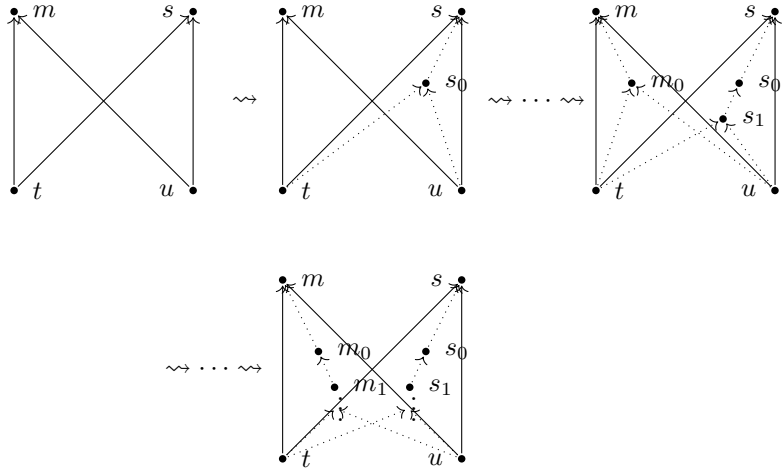
\begin{center}
\begin{tikzpicture}[scale=1.19, label distance=2pt, shorten >=.5pt, shorten <=.5pt]

   \node [fill, circle, inner sep=1pt, label=right:$t$, label=left:{}] (worldu) at (-4,0) {};
   \node [fill, circle, inner sep=1pt, label=left:$s$, label=right:{}] (worldw) at (-2,2) {};
   \node [fill, circle, inner sep=1pt, label=right:$m$, label=left:{}] (worldm) at (-4,2) {};
   \node [fill, circle, inner sep=1pt, label=left:$u$, label=right:{}] (worldv) at (-2,0) {};

   \node [inner sep=0pt, label=right:$\rightsquigarrow$, label=left:{}] (ar1) at (-1.75,1) {};

   \node [fill, circle, inner sep=1pt, label=right:$t$, label=left:{}] (worldu2) at (-1,0) {};
   \node [fill, circle, inner sep=1pt, label=left:$s$, label=right:{}] (worldw2) at (1,2) {};
   \node [fill, circle, inner sep=1pt, label=right:$m$, label=left:{}] (worldm2) at (-1,2) {};
   \node [fill, circle, inner sep=1pt, label=left:$u$, label=right:{}] (worldv2) at (1,0) {};
   
   \node [fill, circle, inner sep=1pt, label=right:$s_0$, label=left:{}] (worldw21) at (.6,1.2) {};

   \node [inner sep=0pt, label=right:$\rightsquigarrow\cdots\rightsquigarrow$, label=left:{}] (ar2) at (1.1,1) {};

   \node [fill, circle, inner sep=1pt, label=right:$t$, label=left:{}] (worldu3) at (2.5,0) {};
   \node [fill, circle, inner sep=1pt, label=left:$s$, label=right:{}] (worldw3) at (4.5,2) {};
   \node [fill, circle, inner sep=1pt, label=right:$m$, label=left:{}] (worldm3) at (2.5,2) {};
   \node [fill, circle, inner sep=1pt, label=left:$u$, label=right:{}] (worldv3) at (4.5,0) {};
   
   \node [fill, circle, inner sep=1pt, label=right:$s_0$, label=left:{}] (worldw31) at (4.1,1.2) {};
   \node [fill, circle, inner sep=1pt, label=right:$m_0$, label=left:{}] (worldm31) at (2.9,1.2) {};
   \node [fill, circle, inner sep=1pt, label=right:$s_1$, label=left:{}] (worldw32) at (3.92,0.8) {};

   \node [inner sep=0pt, label=right:$\rightsquigarrow\cdots\rightsquigarrow$, label=left:{}] (ar2) at (-2.5,-2) {};
   
   \node [fill, circle, inner sep=1pt, label=right:$t$, label=left:{}] (worldu4) at (-1,-3) {};
   \node [fill, circle, inner sep=1pt, label=left:$s$, label=right:{}] (worldw4) at (1,-1) {};
   \node [fill, circle, inner sep=1pt, label=right:$m$, label=left:{}] (worldm4) at (-1,-1) {};
   \node [fill, circle, inner sep=1pt, label=left:$u$, label=right:{}] (worldv4) at (1,-3) {};
   
   \node [fill, circle, inner sep=1pt, label=right:$s_0$, label=left:{}] (worldw41) at (0.6,-1.8) {};
   \node [fill, circle, inner sep=1pt, label=right:$m_0$, label=left:{}] (worldm41) at (-0.6,-1.8) {};
   \node [fill, circle, inner sep=1pt, label=right:$s_1$, label=left:{}] (worldw42) at (0.42,-2.2) {};
   \node [fill, circle, inner sep=1pt, label=right:$m_1$, label=left:{}] (worldm42) at (-0.42,-2.2) {};
   
   \node [inner sep=0pt, label=right:$\mathbf{\vdots}$, label=left:{}] (ar4) at (0.24,-2.35) {};
   \node [inner sep=0pt, label=right:$\mathbf{\vdots}$, label=left:{}] (ar5) at (-0.58,-2.35) {};

  \draw [->] (worldu) to (worldw);
    \draw [->] (worldu) to (worldm);
    \draw [->] (worldv) to (worldw);
  \draw [->] (worldv) to (worldm);

  \draw [->] (worldu2) to (worldw2);
    \draw [->] (worldu2) to (worldm2);
    \draw [->] (worldv2) to (worldw2);
  \draw [->] (worldv2) to (worldm2);
  
  \draw [->, dotted] (worldu2) to (worldw21);
    \draw [->, dotted] (worldv2) to (worldw21);
    \draw [->, dotted] (worldw21) to (worldw2);

    \draw [->] (worldu3) to (worldw3);
    \draw [->] (worldu3) to (worldm3);
    \draw [->] (worldv3) to (worldw3);
  \draw [->] (worldv3) to (worldm3);
  
  \draw [->, dotted] (worldu3) to (worldw32);
    \draw [->, dotted] (worldv3) to (worldw32);
    \draw [->, dotted] (worldw31) to (worldw3);
    \draw [->, dotted] (worldw32) to (worldw31);
    \draw [->, dotted] (worldu3) to (worldm31);
    \draw [->, dotted] (worldv3) to (worldm31);
    \draw [->, dotted] (worldm31) to (worldm3);

    \draw [->] (worldu4) to (worldw4);
    \draw [->] (worldu4) to (worldm4);
    \draw [->] (worldv4) to (worldw4);
  \draw [->] (worldv4) to (worldm4);
  
    \draw [->, dotted] (worldw41) to (worldw4);
    \draw [->, dotted] (worldw42) to (worldw41);
    \draw [->, dotted] (worldm41) to (worldm4);
    \draw [->, dotted] (worldm42) to (worldm41);
    \draw [->, dotted] (worldu4) to (-0.42,-2.45);
    \draw [->, dotted] (worldu4) to (0.38,-2.45);
    \draw [->, dotted] (worldv4) to (-0.38,-2.45);
    \draw [->, dotted] (worldv4) to (0.42,-2.45);

\end{tikzpicture}
\end{center}
The transformation will take place step by step,\footnote{In fact, quite literally, the construction can be seen as an instance of the step-by-step method.} adding one point of the infinite chain at a time (first $s_0$, then $s_1$, and so on). To ensure $\Vdash_S$-preservation, each point in the constructed chain $s_0,s_1,\hdots$  below $s$ will be constructed as to satisfy the same formulas as $s$. For this reason, the basic idea depicted above is, as-such, too naive and runs into the following problems:

\begin{itemize}
    \item[$\bullet$] $s$ (or $m$) might be the supremum of other points, hence merely `duplicating' $s$ does not suffice. Our solution will involve duplicating the full downset, ${\downarrow}s$, and positioning each duplicate directly below its duplicator.
    \item[$\bullet$] When adding, for example, $s_0$, not only does $\{t,u\}$ gain a new upper bound, but so does every $\{y,z\}$ where $y\leq t$ and $z\leq u$. Problem then is that, while $\{t,u\}$ did not have a supremum, $\{y,z\}$ could still have had some supremum $x$. Therefore, to maintain $x$ as the supremum of $\{y,z\}$, it does not suffice to have $s_0$ seen by all points in ${\downarrow}t$ and ${\downarrow}u$: it must also be seen by $x$. Following this line of reasoning to its conclusion, $s_0$ must be seen by all points in the least downset containing $\{t,u\}$ closed under binary suprema.
\end{itemize}
Hopefully, the depiction has illuminated the (naive) spirit of the argument and the bullet points offered intuition for the actual, more complex construction.\footnote{I recommend revisiting the bullet points during or after studying the lemma below.} It will all be made rigorous in the subsequent lemma and proposition, following one final preliminary definition of the concept of a \textit{supremum p-morphism}.\footnote{As condition (1) specifies, we need preservation of $\Vdash_S$, not necessarily of $\Vdash_M$, hence the concept of a supremum p-morphism and not of a m.u.b. p-morphism.}

\begin{definition}
    Let $(P,\leq), (P', \leq')$ be posets. A function
        $$f:P'\to P$$
    is a \textnormal{(supremum) p-morphism} if it satisfies the following conditions:
    \begin{enumerate}[leftmargin=40pt]
        \item [\textnormal{(forth)}] if $x'=\sup'\{y',z'\}$, then $f(x')=\sup\{f(y'),f(z')\}$; and
        \item [\textnormal{(back)}] if $f(x')=\sup\{y,z\}$, then there exist $\{y',z'\}\subseteq P'$ s.t. $f(y')=y, f(z')=z$ and $x'=\sup'\{y',z'\}$.
    \end{enumerate}
\end{definition}
It is readily verified that $\Vdash_S$ is preserved by p-morphic images, at the level of pointed models and at the level of frames: if $(P,\leq)$ is a p-morphic image of $(P',\leq')$ and $(P',\leq')\Vdash_S \varphi$, then $(P,\leq)\Vdash_S \varphi$.\footnote{Take note that this concept of a (supremum) p-morphism deviates from the conventional concept of a p-morphism between posets. This adjustment is necessary because we're concerned with suprema, rather than solely `$\leq$', so to speak.} Thus, it suffices to show that every poset $(P,\leq)$ is the p-morphic image of a poset $(P', \leq')$ satisfying condition (2).


\begin{lemma}[Principal lemma]\label{lm:minDeletion}
    Let $(P,\leq)$ be a poset and $\{s,t,u\}\subseteq P$ be s.t. $s\in\mathrm{mub}\{t,u\}$ but $s\neq \sup\{t,u\}$. Then $(P,\leq)$ is the p-morphic image of a poset $(P', \leq')$ s.t.
    \begin{enumerate}
        \item [\textnormal{1.}] $P\subseteq P', |P'|\leq \max\{\aleph_0, |P|\}$;
        \item [\textnormal{2.}] ${\leq'}\mathrel{\cap} (P\times P)={\leq}$;
        \item [\textnormal{3.}] the witnessing p-morphism $f$ 
        restricts to the identity on $P$, i.e., $f{\upharpoonright} P=\mathrm{Id}_P$;
        \item [\textnormal{4.}] if $x=\sup\{y,z\}$, then $x=\sup'\{y,z\}$;
        \item [\textnormal{5.}] $s\notin \mathrm{mub}'\{t,u\}$.
    \end{enumerate}
\end{lemma}
\begin{proof}
    We let 
        $$P'\mathrel{:=}P\sqcup {\downarrow}s=\{(x,0), (y,1)\mid x\in P, y\in{\downarrow}s\},$$
    where ${\downarrow} s\mathrel{:=}\{s'\mid s'\leq s\}$, and let
        $$f:P' \to P$$
    be the function given by
        $$f(x,i)=x$$
    for all $x\in P, i\in \{0,1\}$. 
    
    To define the relation $\leq'$, we need an auxiliary definition of $\mathcal{G}(t,u)\subseteq P$ as the least downset containing $\{t,u\}$ and closed under binary suprema. That is, 
        $$\mathcal{G}(t,u)=\bigcup_{n\in\omega}\mathcal{G}_n(t,u),$$ 
    where 
    
        $$\mathcal{G}_0(t,u)={\downarrow}t\mathrel{\cup} {\downarrow} u$$
    and
        $$\mathcal{G}_{n+1}(t,u)={\downarrow}\big(\mathcal{G}_n(t,u)\cup \big\{\sup\{b_n, c_n\}\mid \{b_n, c_n\}\subseteq \mathcal{G}_n(t,u)\big\}\big).$$
    Then for all $(x,i), (y,j)\in P'$, we let $(y,j)\leq'(x,i)$ \textbf{iff}
     \begin{enumerate}[label=(\roman*), leftmargin=40pt]
        \item $i=0$ and $y\leq x$, \textbf{or}
        \item $j=i=1$ and $y\leq x$, \textbf{or}
        \item $j=0, i=1, y\in \mathcal{G}(t,u)$, and $x=s$. 
    \end{enumerate}   
    Having defined $P', \leq',$ and $f$, we claim that they witness the conditions of the lemma, namely:
    \begin{enumerate}[leftmargin=40pt]
        \item [(a)] $(P', \leq')$ forms a poset; 
        \item[(b)] conditions 1.-5. of the lemma are met; and
        \item[(c)] $f$ constitutes an onto p-morphism. 
    \end{enumerate}
    To prove this, we will need the following assertions:
    \begin{itemize}[leftmargin=40pt]
        \item if $y\in \mathcal{G}(t,u)$, then $y<s$; and
        \item $f$ is order-preserving, i.e., $(y,j)\leq' (x,i)$ implies $y\leq x$.
    \end{itemize}
    To show the former, note that since $\mathrm{mub}\{t,u\}\ni s\neq \sup\{t,u\}$, there must be some $m\geq t,u$ s.t. $m$ and $s$ are incomparable. Using this, we show that for all $n\in\omega$, $\mathcal{G}_n(t,u)\subsetneq {\downarrow}s$ and $\mathcal{G}_n(t,u)\subsetneq {\downarrow}m$, which suffices for establishing this assertion. 
    
    Since $s$ and $m$ are incomparable upper bounds, we know that $d<e$ for $d\in \{t,u\}$ and $e\in \{s,m\}$, which shows the base case. Accordingly, assume for some $n\in \omega$ that $\mathcal{G}_n(t,u)\subsetneq {\downarrow}s$ and $\mathcal{G}_n(t,u)\subsetneq {\downarrow}m$, and let $\{b_n, c_n\}\subseteq \mathcal{G}_n(t,u)$ be arbitrary s.t. $\sup\{b_n, c_n\}$ exists. It then suffices to show that $\sup\{b_n, c_n\}<s$ and $\sup\{b_n, c_n\}<m$. Since both $s$ and $m$ are upper bounds of $\{b_n, c_n\}$ by the IH, we have that $\sup\{b_n, c_n\}\leq s$ and $\sup\{b_n, c_n\}\leq m$. And because $s$ and $m$ are incomparable, the inequalities must be strict---finalizing our proof of the former assertion. 
    
    Next, to show the latter auxiliary assertion, assume $(y,j)\leq'(x,i)$. If $i=0$ or $j=i=1$, we have that $y\leq x$, showing the required. And if $j=0, i=1, y\in \mathcal{G}(t,u)$, and $x=s$, then, by the former assertion, $y< s=x$, which covers the last case.
    \\\\
     With these assertions proven, we're ready to tackle (a), (b), and (c), beginning with \textbf{(a)}: showing that $\leq'$ is a partial order.

    \begin{itemize}
        \item \textit{Reflexivity} follows by reflexivity of $\leq$ and (i), (ii). 
        \item \textit{Transitivity}: Suppose $(z,k)\leq'(y,j)\leq'(x,i)$. By order-preservation, $z\leq y\leq x$, so, since $\leq$ is transitive, $z\leq x$. Thus, if $i=0$ or $k=i=1$, then $(z,k)\leq'(x,i)$ as required. And if $k=0$ and $i=1$, we must show that 
            $$z\in \mathcal{G}(t,u) \quad \text{and}\quad x=s.$$
        If $j=0$, then since $(y,j)\leq'(x,i)$, we have that $y\in \mathcal{G}(t,u)$ and $x=s$, so because $z\leq y$ and $\mathcal{G}(t,u)$ is a downset, we also have $z\in \mathcal{G}(t,u)$. Lastly, if $j=1$, then since $(z,k)\leq'(y,j)$, we have that $z\in \mathcal{G}(t,u)$ and $y=s$, so since $y\leq x$ and $x\in {\downarrow}s$, it must also be that $x=s$.
        \item \textit{Anti-symmetry}: Suppose $(y,j)\leq' (x,i)$ and $(x,i)\leq' (y,j)$. If $j=i$, we're done by anti-symmetry of $\leq$. Moreover, we cannot have $j\neq i$, since if, say, $j=0, i=1$, then $y\in \mathcal{G}(t,u)$ and $x=s$, so $y<s=x$, but by order-preservation we also have that $x\leq y$ -- contradiction.
    \end{itemize}
Having established (a), we proceed to verify \textbf{(b)}, namely that 1.-5. are met. 

Note that our definition of $P'$ as the \textit{disjoint} union of $P$ and ${\downarrow}s$ isn't mathematically substantive but chosen for ease of notation; we could just as well have defined $P'$ as an actual set-theoretical extension of $P$. With this in mind, it is easy to see that 1. through 3. are fulfilled. To clarify, establishing conditions 4. and 5. then amounts to the following:
\begin{enumerate}
    \item [4.]if $x=\sup\{y,z\}$, then $(x,0)=\sup'\{(y,0), (z,0)\}$;
    \item [5.]$(s,0)\notin\mathrm{mub}'\{(t,0), (u,0)\}$
\end{enumerate}
For the latter, observe that $(s,1)\leq'(s,0)$, and since $\{t,u\}\subseteq \mathcal{G}(t,u)$, we also have $(t,0),(u,0)\leq'(s,1)$; hence, $(s,0)\notin\mathrm{mub}'\{(t,0), (u,0)\}$. 

Concerning the former, if $x=\sup\{y,z\}$, then $(x,0)$ is an upper bound of $\{(y,0), (z,0)\}$. To show it is the least upper bound, suppose $(v,i)$ is another upper bound of $\{(y,0), (z,0)\}$. We must then demonstrate that $(x,0)\leq'(v,i)$. By order-preservation and the assumption that $x=\sup\{y,z\}$, we know $x\leq v$. If $i=0$, we are done. If $i=1$, we need to show that 
    $$x\in \mathcal{G}(t,u) \quad \text{and}\quad  v=s.$$ 
Since $(v,1)$ is an upper bound of $\{(y,0), (z,0)\}$, we have $\{y,z\}\subseteq \mathcal{G}(t,u)$ and $v=s$. Consequently, as $\mathcal{G}(t,u)$ is closed under binary suprema and $x=\sup\{y,z\}$ by assumption, we also have $x\in \mathcal{G}(t,u)$, which completes the proof of (b).\footnote{Notice how condition 4. is met by reason of our definition of $\mathcal{G}(t,u)$, cf. the latter of the two bullet points preceding this lemma.}\\\\
It remains to verify \textbf{(c)}: that $f$ constitutes an onto p-morphism. It is certainly an onto function, so we need only show that the back and forth clauses hold. Let's address each in turn.

\textit{Forth.} Suppose $(x,i)=\sup'\{(y,j),(z,k)\}$. By order-preservation, $x$ is an upper bound of $\{y,z\}$. So assume that $v$ is another upper bound. Then $(v,0)$ is an upper bound of $\{(y,j),(z,k)\}$, hence $(x,i)\leq' (v,0)$, so by another application of order-preservation, we get that $x\leq v$, which exactly shows that $f(x,i)=x=\sup\{y,z\}=\sup'\{f(y,j),f(z,k)\}$.

\textit{Back.} Suppose $f(x,i)=x=\sup\{y,z\}$ for some $y,z\in P$. Going by cases, we get:
\begin{itemize}
    \item If $i=0$, we have, by 4., that $(x,i)=\sup'\{(y,0), (z,0)\}$. This suffices as $f(y,0)=y$ and $f(z,0)=z$.
    \item If $i=1$, then $x\in {\downarrow}s$, implying $\{y,z\}\subseteq{\downarrow}s$. Consequently, $\{(y,1), (z,1)\}\subseteq P'$. Moreover, $f(y,1)=y, f(z,1)=z$. Therefore, we need only verify that $(x,1)=\sup'\{(y,1), (z,1)\}$. Since $x=\sup\{y,z\}$, $(x,1)$ is an upper bound of $\{(y,1), (z,1)\}$. So suppose $(v,j)$ also is an upper bound of $\{(y,1), (z,1)\}$. Then $v$ is an upper bound of $\{y,z\}$, hence $x\leq v$. Thus, $(x,1)\leq'(v,j)$, as required.
\end{itemize}
This completes our proof of \textbf{(c)} and, consequently, of the lemma.
\end{proof}

With this key lemma at our disposal, we can deduce the succeeding results.

\begin{proposition}\label{pr:p-mMIN}
    Every poset $(P,\leq)$ is the p-morphic image of a poset $(P', \leq')$ satisfying
    \begin{align*}
        &(2) &&\forall s', t', u'\in P': \qquad s'\in\mathrm{mub}\{t', u'\} && \Leftrightarrow && s'=\sup\{t', u'\}
    \end{align*}
\end{proposition}

\begin{proof}
     The proposition follows through a repeated application of the preceding lemma. 
     
     To begin, for ease of notation, we assume $(P, \leq)$ is countable.\footnote{This is without loss of generality since the general case follows through transfinite recursion, a theoretically straightforward but notationally cumbersome adjustment.}
     
     Next, we enumerate all triples $(s,t,u)\in (P\cup A)$, where $A$ is any fixed denumerably infinite set disjoint from $P$. 

     Then, we proceed iteratively. Starting with $(P_0, \leq_0)\mathrel{:=}(P, \leq)$, we generate $(P_{n+1}, \leq_{n+1})$ from $(P_{n}, \leq_{n})$ by applying the preceding lemma to the least tuple $(s,t,u)\in (P\cup A)$ from our enumeration s.t. $s\in \mathrm{mub}_{n}\{t,u\}$ but $s\neq\sup_{n}\{t,u\}$. By condition 1. of the lemma, the added points $P_{n+1}\setminus P_n$ can be drawn as new points from our auxiliary stock $A$. We denote the associated onto p-morphisms by $f_{n+1}: P_{n+1}\to P_n$, and define the composed functions derived from these as $g_{n+1} \mathrel{:=} f_1 \circ \cdots \circ f_{n+1}: P_{n+1} \to P$. Recall that p-morphisms are closed under composition, so each $g_n$ is also a p-morphism.


     Lastly, defining        
        $$P_\omega\mathrel{:=}\bigcup_{n\in\omega}P_n, \qquad \leq_\omega\mathrel{:=}\bigcup_{n\in\omega}\leq_n,$$
     we claim: 
     \begin{enumerate}
         \item[(a)] $(P_\omega, \leq_\omega)$ forms a poset;
        \item [(b)] 
        $g_\omega\mathrel{:=}\bigcup_{n\in \omega}g_n$ is an onto p-morphism from $(P_\omega, \leq_\omega)$ to $(P, \leq)$; and
        \item [(c)]$(P_\omega, \leq_\omega)$ fulfills clause (2), i.e., 
        \begin{align*}
        &(2) &&\forall s', t', u'\in P_\omega: \qquad s'\in\mathrm{mub}\textit{}_{\omega}\{t', u'\} && \Leftrightarrow && s'=\sup\textit{}_{\omega}\{t', u'\}
    \end{align*}
     \end{enumerate}    
    (a) is straightforward. To establish (b), we note that $g_\omega$ is a function because of 3. of the foregoing lemma. It is evidently onto. The forth condition follows using 2. of the lemma and the forth conditions of the $g_n$s, while the back condition follows using the back condition of the $g_n$s and 4. of the lemma. Lastly, (c) follows by the enumeration of triples and 2. and 5. of the lemma.
\end{proof}
Finally, we can deduce our main theorem and corollaries:

\begin{theorem}\label{maintheorem}
    $\mathbf{MIN}=\mathbf{MIL}$
\end{theorem}
\begin{proof}
    We have demonstrated that $\mathbf{MIL}\subseteq \mathbf{MIN}$ in Theorem \ref{Soundness}. The converse inclusion is an immediate consequence of the proposition above, as explicated in the \textit{Proof Strategy} paragraph.
\end{proof}

\begin{corollary}\label{MainCorollary}
    $\mathbf{MIN}$ is decidable, and  Theorem \ref{MILAx} provides an axiomatization for it.
\end{corollary}
\begin{proof}
    This follows from Theorem \ref{maintheorem} and the axiomatization and decidability proofs for $\mathbf{MIL}$ achieved in \cite{SBKMIL}.
\end{proof}

\section{Preorders and the Residual Implication}\label{sec:Extensions}
With our main results established, this final section investigates extensions by exploring the implications of considering preorders as our class of frames, rather than the subclass of partial orders, and by enriching the languages with the residual implications of $\Min$ and $\Sup$, respectively. We commence with the former.

\subsubsection{Preorder Frames}
    Thus far, our study has focused on $\mathbf{MIN}$ and $\mathbf{MIL}$ defined over posets. Making this explicit, we write
        $$\mathbf{MIN}_\textit{Pos}\mathrel{:=}\mathbf{MIN}, \qquad \mathbf{MIL}_\textit{Pos}\mathrel{:=}\mathbf{MIL}.$$
    However, the concepts of least and minimal upper bounds extend to general preorders, where they are known as \textit{quasi}-least and \textit{quasi}-minimal upper bounds. This motivates the study of the corresponding logics: 
        $$\mathbf{MIL}_\textit{Pre} \qquad\text{and} \qquad \mathbf{MIN}_\textit{Pre}.$$
    Indeed, in the literature, $\mathbf{MIL}_\textit{Pre}$ has been treated on a par with $\mathbf{MIL}_\textit{Pos}$, and \cite{SBKMIL} showed that, at the level of validities and consequences, they coincide: $\mathbf{MIL}_\textit{Pre}=\mathbf{MIL}_\textit{Pos}$. 

    Analogously, as a consequence of Theorem \ref{maintheorem} and Corollary \ref{MainCorollary}, we can show that in our minimal-upper-bound setting:
        $$\mathbf{MIN}_\textit{Pre}=\mathbf{MIN}_\textit{Pos}.$$
    One inclusion
        $$\mathbf{MIN}_\textit{Pre}\subseteq\mathbf{MIN}_\textit{Pos}$$
    follows because the latter logic is defined by restricting the class of frames of the former logic, while the other holds as the axiomatization of Theorem \ref{MILAx} is sound w.r.t. $\mathbf{MIN}_\textit{Pre}$.    

\subsubsection{The Residual Implication} Another direction worth exploring involves enriching the language with the implication `\textbackslash' with semantics:
\begin{align*}
        &\mathbb{M}, u\Vdash \varphi\text{\textbackslash}\psi && \textbf{iff} && \text{for all $t,s\in P$, if } \mathbb{M}, t\Vdash \varphi \text{ and $s\in \mathrm{mub}\{t,u\}$,}\text{ then } \mathbb{M}, s\Vdash \psi.
\end{align*}
We denote the resulting modal logic of minimal upper bounds as $\mathbf{MIN}(\textbackslash)$. This type of semantic clause dates back to \cite{Urquhart72} and \cite{Urquhart73}, along with its informational interpretation: an information state $u$ satisfies $\varphi\text{\textbackslash}\psi$ iff
for all information states $t$ and all fusions $s\in \mathrm{mub}\{t,u\}$, whenever $t$ satisfies the antecedent $\varphi$, the fusion $s$ satisfies the consequent $\psi$; in other words, any way of adding the information $\varphi$ to $u$ results in $\psi$.

Note that `$\textbackslash$' is residuated by `$\Min$', standing to `$\Min$' as, for instance, the Lambek residuals stand to the Lambek product (hence the choice of the symbol `\textbackslash') and the relevant implication to the intensional conjunction. In the context of modal information logic, $\mathbf{MIN}(\textbackslash)$'s counterpart $\mathbf{MIL}(\textbackslash)$ was introduced in \cite{JvB2021} and studied in \cite{SBKMIL}, where it was proven decidable and axiomatized as follows:\footnote{Naturally, in $\mathbf{MIL}(\textbackslash)$, the semantic clause for `\textbackslash' refers to the supremum relation instead of the m.u.b. relation.}

\begin{theorem}[Axiomatization of $\mathbf{MIL}(\textbackslash)$ \cite{SBKMIL}]\label{def:axInfImp}
    $\mathbf{MIL}(\textbackslash)$ is sound and complete with respect to the least set of formulas $\mathbf{L}$ in the extended language with $\textbackslash$ satisfying the following conditions:

    \begin{enumerate}
        \item [(i)] $\mathbf{L}$ is closed under the axioms and rules of the axiomatization in Theorem \ref{MILAx}.
        \item [(ii)] $\mathbf{L}$ includes the $\mathbf{K}$-axioms for $\textbackslash$.
        \item [(iii)] $\mathbf{L}$ includes the axioms:
            \begin{enumerate}
                \item [(L1)] $\Sup p(p\textbackslash q)\to q$,
                \item [(L2)] $p\to q\textbackslash (\Sup pq)$.
            \end{enumerate}
        \item [(iv)] $\mathbf{L}$ is closed under the left-necessitation rule for $\textbackslash$, i.e.:
            \begin{enumerate}
                \item [($N$)] If $\varphi\in \mathbf{L}$, then $\psi \textbackslash \varphi \in \mathbf{L}$.
            \end{enumerate}
    \end{enumerate}
\end{theorem}

With this axiomatization at hand, we can extend the proof of Theorem \ref{maintheorem} to the logics $\mathbf{MIN}(\textbackslash)$ and $\mathbf{MIL}(\textbackslash)$ in the augmented language.

\begin{theorem}
    $\mathbf{MIN}(\textbackslash)=\mathbf{MIL}(\textbackslash)$.
\end{theorem}
\begin{proof}
    Once again, the inclusion $\mathbf{MIL}(\textbackslash)\subseteq \mathbf{MIN}(\textbackslash)$ follows by a routine check that the aforementioned axiomatization is sound w.r.t. the semantics of $\mathbf{MIN}(\textbackslash)$.
    
    For the converse inclusion $\mathbf{MIN}(\textbackslash)\subseteq \mathbf{MIL}(\textbackslash)$, examining the proofs of Theorem \ref{maintheorem}, Proposition \ref{pr:p-mMIN}, and Lemma \ref{lm:minDeletion}, it is readily seen that it suffices to show that the p-morphism $f$ of Lemma \ref{lm:minDeletion} additionally satisfies the back condition for $\textbackslash$:
    \begin{enumerate}[leftmargin=40pt]
        \item [($\text{\textbackslash}$-back)] if $x=\sup\{f(y'),z\}$, then there exist $\{x',z'\}\subseteq P'$ s.t. $f(x')=x, f(z')=z$ and $x'=\sup'\{y',z'\}$.
    \end{enumerate}
This is because the forth condition for $\textbackslash$ coincides with that for $\Sup$; thus, if $f$ additionally satisfies the back condition for $\textbackslash$, it preserves $\Vdash_S$ even in the extended language, as can be easily verified.
    
    Accordingly, suppose $x=\sup\{f(y,j),z\}$ for some $x,z\in P$. Going by cases, we get:
\begin{itemize}
    \item If $j=0$, we have, by 4. in Lemma \ref{lm:minDeletion}, that $(x,0)=\sup'\{(y,0), (z,0)\}$ -- showing the desired.
    \item If $j=1$ and $x\in {\downarrow}s$, then also $z \in{\downarrow}s$. Thus, $\{(x,1), (z,1)\}\subseteq P'$, and $(x,1)$ is an upper bound of $\{(y,1), (z,1)\}$. So suppose $(v,i)$ also is an upper bound of $\{(y,1), (z,1)\}$. It then suffices to show that $(x,1)\leq' (v,i)$. Since $f$ is order-preserving, we have that $y\leq v\geq z$, so by assumption $x\leq v$, hence $(x,1)\leq' (v,i)$.
    
    \item If $j=1$ and $x\notin {\downarrow}s$, then we claim that $(x,0)=\sup'\{(y,1), (z,0)\}$. Since $x=\sup\{f(y,j),z\}$ by assumption, $(x,0)$ is an upper bound of $\{(y,1), (z,0)\}$. Assume $(v,i)$ also is an upper bound of $\{(y,1), (z,0)\}$. Then $y\leq v\geq z$, so $x\leq v$. Thus, $v \notin {\downarrow}s$, hence we must have $i=0$, whence $(x,0)\leq' (v,i)$, as required.
\end{itemize}
    Thus, we have demonstrated that the function $f$ defined in Lemma \ref{lm:minDeletion} satisfies the $\textbackslash$-back condition, allowing us to conclude that Lemma \ref{lm:minDeletion}, Proposition \ref{pr:p-mMIN}, and Theorem \ref{maintheorem} all extend to our present richer setting.
\end{proof}

Likewise, the analogue of corollary \ref{MainCorollary} follows:

\begin{corollary}
    $\mathbf{MIN}(\textbackslash)$ is decidable, and Theorem \ref{def:axInfImp} provides an axiomatization for it.
\end{corollary}

\subsubsection{Summary} In conclusion, we have proven that $\mathbf{MIN}=\mathbf{MIL}$, 
and even that
    $$\mathbf{MIN}\mathrel{:=}\mathbf{MIN}_\textit{Pos}=\mathbf{MIN}_\textit{Pre}=\mathbf{MIL}_\textit{Pre}=\mathbf{MIL}_\textit{Pos}\mathrel{=:}\mathbf{MIL}$$
Likewise, our proof that $\mathbf{MIN}(\textbackslash)=\mathbf{MIL}(\textbackslash)$ is readily extended to yield: 
    $$\mathbf{MIN}(\textbackslash)\mathrel{:=}\mathbf{MIN}_\textit{Pos}(\textbackslash)=\mathbf{MIN}_\textit{Pre}(\textbackslash)=\mathbf{MIL}_\textit{Pre}(\textbackslash)=\mathbf{MIL}_\textit{Pos}(\textbackslash)\mathrel{=:}\mathbf{MIL}(\textbackslash)$$
In broader terms, one might therefore conclude that the landscape of modal logics of minimal and least upper bounds is uniform.

However, even if the least and minimal-upper-bound interpretations are hard to tell apart from a modal perspective, our central proof method for Theorem \ref{maintheorem}, relying on the construction of infinite chains, suggests a notable context where they  do differ: on \textit{finite} preorders/posets. This is witnessed by the formula
        $$(\Past p \land \Past q)\to \Past(\Min pq),$$
which is valid on finite frames under the minimal-upper-bound interpretation but not under the least-upper-bound interpretation.


\begin{credits}
\subsubsection{\ackname} This paper is based on Chapter 5 of my master’s thesis \cite{thesis}, supervised
by Johan van Benthem and Nick Bezhanishvili. I am grateful to both for their guidance throughout the thesis process. I acknowledge partial
support from Nothing is Logical (NihiL), an NWO OC project (406.21.CTW.023).

\subsubsection{\discintname}
The author(s) have no competing interests to declare.
\end{credits}

\bibliographystyle{splncs04} 
\bibliography{mybibliography}

\begin{thebibliography}{10}
\providecommand{\url}[1]{\texttt{#1}}
\providecommand{\urlprefix}{URL }
\providecommand{\doi}[1]{https://doi.org/#1}

\bibitem{Aloni22}
Aloni, M.: Logic and conversation: The case of free choice. Semantics and Pragmatics  \textbf{15}(5),  1--60 (Jun 2022)

\bibitem{AndersonBelnap75}
Anderson, A.R., Belnap, N.D.: Entailment: The Logic of Relevance and Necessity, Vol. I. Princeton University Press (1975)

\bibitem{JvB1996}
van Benthem, J.: Modal logic as a theory of information. In: Copeland, J. (ed.) Logic and Reality. Essays on the Legacy of Arthur Prior, pp. 135--168. Clarendon Press, Oxford (1996)

\bibitem{JvB2019}
van Benthem, J.: Implicit and explicit stances in logic. Journal of Philosophical Logic  \textbf{48}(3),  571--601 (2019). \doi{10.1007/s10992-018-9485-y}

\bibitem{JvB2021}
van Benthem, J.: Relational patterns, partiality, and set lifting in modal semantics. In: Weiss, Y., Birman, R. (eds.) `Saul Kripke on Modal Logic' in the series `Outstanding Contributions to Logic'. Springer Cham. \\Preprint available at:\textit{ } https://eprints.illc.uva.nl/id/eprint/1773/ (Forthcoming)

\bibitem{Buszkowski21}
Buszkowski, W.: Lambek Calculus with Classical Logic, pp. 1--36 (03 2021). \doi{10.1007/978-3-030-63787-3_1}

\bibitem{Fine17}
Fine, K.: Truthmaker semantics. In: Hale, B., Wright, C., Miller, A. (eds.) A Companion to the Philosophy of Language, chap.~22, pp. 556--577. John Wiley \& Sons, Ltd (2017). \doi{https://doi.org/10.1002/9781118972090.ch22}

\bibitem{Fraassen69}
van Fraassen, B.C.: Facts and tautological entailments. The Journal of Philosophy  \textbf{66},  477--487 (1969)

\bibitem{thesis}
Knudstorp, S.B.: Modal Information Logics. Master's thesis, Institute for Logic, Language and Computation, University of Amsterdam (2022)

\bibitem{SBKTML}
Knudstorp, S.B.: Logics of truthmaker semantics: comparison, compactness and decidability. Synthese  \textbf{202}(206) (2023). \doi{10.1007/s11229-023-04401-1}

\bibitem{SBKMIL}
Knudstorp, S.B.: Modal information logics: Axiomatizations and decidability. Journal of Philosophical Logic pp. 1723--1766 (2023). \doi{10.1007/s10992-023-09724-5}

\bibitem{Lambek58}
Lambek, J.: The mathematics of sentence structure. American Mathematical Monthly  \textbf{65},  154--170 (1958)

\bibitem{Urquhart72}
Urquhart, A.: Semantics for relevant logics. Journal of Symbolic Logic  \textbf{37},  159--169 (1972)

\bibitem{Urquhart73}
Urquhart, A.: The Semantics of Entailment. Ph.D. thesis, University of Pittsburgh (1973)

\bibitem{Yang16}
Yang, F., Väänänen, J.: Propositional logics of dependence. Annals of Pure and Applied Logic  \textbf{167}(7),  557--589 (jul 2016). \doi{10.1016/j.apal.2016.03.003}

\end{thebibliography}

\end{document}